\documentclass[preprint,12pt]{elsarticle}

\usepackage{amsmath}
\usepackage{amsfonts}
\usepackage{amssymb}
\usepackage{amsthm}
\usepackage{booktabs}
\usepackage{ragged2e}
\usepackage[utf8]{inputenc}
\usepackage[english]{babel}
\usepackage{graphicx}
 \usepackage{parskip}
\usepackage{textpos}
\usepackage[utf8]{inputenc}
\usepackage[T1]{fontenc}
\usepackage{helvet}
\usepackage{amsfonts}
\usepackage{amsmath}
\usepackage{amssymb}
\usepackage{tikz, graphicx}
\usepackage{hyperref}
\biboptions{numbers,sort&compress}
\usepackage{bigints}
\usepackage{xcolor}

\newtheorem{theorem}{Theorem}[section]

\newtheorem{corollary}[theorem]{Corollary}

\newtheorem{lemma}[theorem]{Lemma}
\newtheorem{example}[theorem]{Example}
\newtheorem{remark}[theorem]{Remark}

\usepackage{float}
\usepackage{graphicx} 
\numberwithin{equation}{section}

\begin{document}
\begin{frontmatter}



\title{A Universal Approximation Theorem for Neural Networks with Outputs in Locally Convex Spaces} 

\author{Sachin Saini}


\begin{abstract}
In this paper, a universal approximation theorem (UAT) for shallow neural networks whose inputs belong to a topological vector space (TVS) and whose outputs take values in a Hausdorff locally convex TVS is established. The networks are constructed using scalar activation functions applied to continuous linear functionals of the input, while the output coefficients lie in the target space. It is shown that this class of neural networks is dense in the space of continuous mappings from a compact subset of the input space into the target space with respect to the topology of uniform convergence induced by the defining seminorms. The result extends existing scalar-valued approximation theorems on TVS and includes Banach and Hilbert-valued approximation as special cases. Several corollaries and examples are provided, illustrating applications to operator approximation between function spaces.
\end{abstract}

\end{frontmatter}

\section{Introduction}

The capability of neural networks (NNs) to approximate functions has become a fundamental topic in both modern approximation theory and machine learning. Beyond their empirical success in regression and classification tasks, NNs provide flexible parametric families capable of representing highly complex functional relationships. From a mathematical viewpoint, their expressive power is captured by various forms of the universal approximation property.

In the classical setting, one considers shallow feedforward NNs acting on finite-dimensional Euclidean spaces. Given an input vector $s \in \mathbb{R}^d$, a network with a single hidden layer produces functions of the form
\[
s \longmapsto \sum_{j=1}^{m} a_j \, \eta(\langle w_j , s \rangle - b_j),
\]
where $w_j \in \mathbb{R}^d$, $a_j, b_j \in \mathbb{R}$, and $\eta : \mathbb{R} \to \mathbb{R}$ is a fixed activation function. Fundamental results due to Cybenko, Funahashi, Hornik, Pinkus, Leshno and others show that, under mild non-degeneracy conditions on $\eta$, such networks are dense in $C(E)$ for every compact set $E \subset \mathbb{R}^d$ \cite{Cybenko1989,Funahashi1989,Hornik1991,Pinkus1999,leshno1993multilayer}. Consequently, any continuous real-valued function on $E$ can be approximated uniformly by shallow NNs.

A key observation underlying these results is that the mapping $s \mapsto \langle w , s \rangle$ is simply a continuous linear functional on $\mathbb{R}^d$. This perspective naturally suggests replacing finite-dimensional inner products by continuous linear functionals on more general spaces. Such an approach allows NNs architectures to be defined on abstract TVS, thereby extending approximation theory beyond the Euclidean framework.

Motivated by this idea, several authors have investigated NNs acting on infinite-dimensional spaces. In Banach space settings, ridge function techniques have been used to establish density results for scalar-valued approximation \cite{sun1992fundamentality,light1992ridge}. NNs have also been employed to approximate nonlinear operators between spaces of functions \cite{chen1995universal}. More recently, operator-learning architectures such as DeepONet \cite{lu2021learning} have stimulated renewed interest in approximation results for mappings between infinite-dimensional spaces. Quantitative analyses in Hilbert and Banach space settings were subsequently developed in \cite{lanthaler2022error,korolev2022two}. Related developments include UAT for vector and hypercomplex-valued NNs over finite-dimensional algebras \cite{valle2024universal}.

A particularly flexible framework was introduced by Ismailov \cite{ismailov2026universal}, who considered feedforward NNs defined on a TVS $S$ with a sufficiently rich continuous dual $S^*$. In that setting, each hidden neuron evaluates a functional $\ell \in S^*$ at the input $s \in S$, followed by a scalar activation $\eta(\ell(s)-\theta)$. Under suitable assumptions on $S$, such as the Hahn-Banach extension property (HBEP), and on the activation function $\eta$, these scalar-valued networks are dense in $C(E)$ for compact subsets $E \subset S$. This result unifies finite and infinite-dimensional input theories within a common functional-analytic framework. However, the approximation result in \cite{ismailov2026universal} is restricted to scalar-valued mappings, that is, functions $F: S \to \mathbb{R}$. The extension of the UAT to NNs with outputs in general TVS, therefore, remains an important open direction.

However, many applications in contemporary analysis and scientific computing require the approximation of mappings whose values lie in infinite-dimensional function spaces. Typical examples include solution operators of differential equations, parameter-to-state maps, function-to-function regression problems, and models whose outputs are functions or distributions. In such situations, one seeks to approximate mappings of the form
\[
\mathcal{T} : S \longrightarrow T,
\]
where $T$ is not merely $\mathbb{R}$ but a Hausdorff LC-TVS. Convergence in such spaces is determined by families of seminorms rather than a single norm, which introduces additional analytical challenges.

The objective of the present work is to develop a UAT for NNs whose inputs belong to a TVS and whose outputs lie in an LC-TVS. The architecture we consider retains a scalar activation function $\eta : \mathbb{R} \to \mathbb{R}$, while allowing the output coefficients to take values in $T$. The resulting class of functions has the form
\[
s \longmapsto \sum_{j=1}^{m} \eta\big(\ell_j(s)-\theta_j\big) \, v_j,
\qquad \ell_j \in S^*, \; v_j \in T.
\]
Such representations may be interpreted as finite-rank nonlinear operator approximations illustrated in Figure~\ref{fig:nn_architecture}. The NNs consists of
a single hidden layer activated by the function $\eta$.

\begin{figure}[htbp]
\centering
\includegraphics[width=0.8\textwidth]{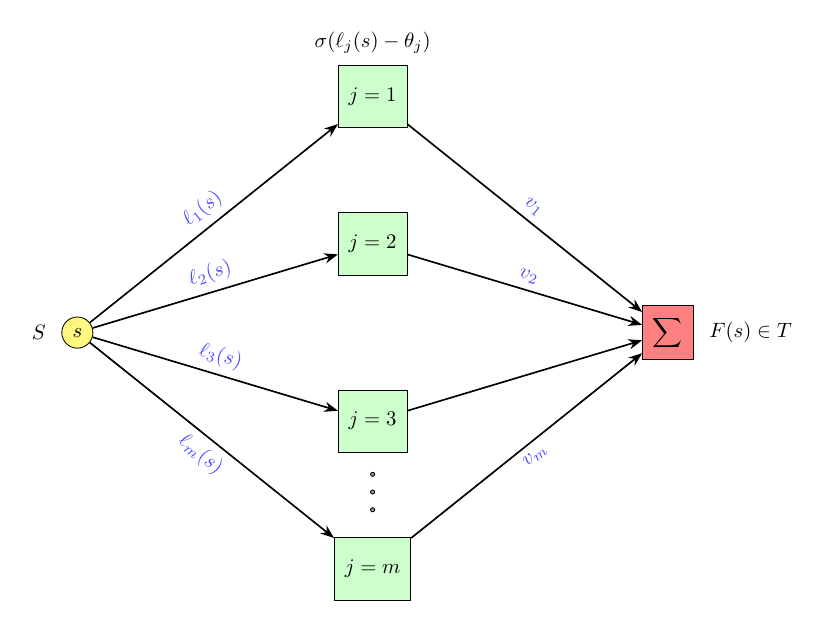}
\caption{Architecture of the above-defined NNs.}
\label{fig:nn_architecture}
\end{figure}

Our main result shows that, under suitable assumptions on $S$, $T$, and $\eta$, this class is dense in $C(E;T)$ for every compact subset $E \subset S$, where density is understood with respect to uniform convergence induced by the defining seminorms of $T$. In particular, when $T=\mathbb{R}$ we recover the scalar-valued theorem of \cite{ismailov2026universal}, while if $T$ is a Banach space the seminorm topology reduces to the standard uniform norm topology. Hence, the Banach-valued case appears as a special instance of a more general LC-approximation theory.

The results obtained here provide a functional-analytic foundation for shallow NN models operating between infinite-dimensional spaces and extend UAT to a broad class of LC-spaces.

\section{Main Results}

In this section, we formulate and prove a UAT for NNs whose inputs lie in a TVS and whose outputs take values in an LC-TVS. The result provides a vector-valued extension of the scalar approximation framework previously developed for NNs defined on TVS.

Throughout this section, let $S$ be a real TVS and let $T$ be a Hausdorff LC-TVS. 
For a compact set $E \subset S$, we denote by $C(E;T)$ the space of continuous mappings from $E$ into $T$, endowed with the topology of uniform convergence induced by the family of continuous seminorms $\rho$ of $T$.

We consider nonlinear approximants of the form
\[
s \longmapsto \sum_{j=1}^{m} 
\eta\big(\ell_j(s) - \theta_j\big)\, v_j,
\qquad 
\ell_j \in S^*, \; \theta_j \in \mathbb{R}, \; v_j \in T,
\]
where $S^*$ denotes the continuous dual of $S$ and 
$\eta : \mathbb{R} \to \mathbb{R}$ is a fixed activation function.

\begin{theorem}[Vector-valued UAT]
\label{thm:vector_uat}
Assume that $S$ possesses the HBEP and 
let $E \subset S$ be compact. 
Let $\eta : \mathbb{R} \to \mathbb{R}$ be continuous and not a polynomial 
on any nonempty open interval.

Define
\[
\mathcal{A}_{S,T}^{\eta}
:=
\mathrm{span}
\left\{
s \mapsto \eta\big(\ell(s)-\theta\big)\, v
:
\ell \in S^*,\ \theta \in \mathbb{R},\ v \in T
\right\}.
\]

Then $\mathcal{A}_{S,T}^{\eta}$ is dense in $C(E;T)$ with respect to 
uniform convergence on $E$ induced by the continuous seminorms of $T$. 
Equivalently, for every $F \in C(E;T)$, every continuous seminorm $\rho$ on $T$, and $\forall$ $\varepsilon > 0$, $\exists$ $G \in \mathcal{A}_{S,T}^{\eta}$ such that
\[
\sup_{s \in E} \rho\big(F(s) - G(s)\big) < \varepsilon .
\]
\end{theorem}

\textbf{Well-definedness of the neural network representation:}
We justify that the NN expressions considered in this paper are well defined. 
Let $s \in S$, $\ell \in S^*$, $\theta \in \mathbb{R}$, and $v \in T$.

\begin{enumerate}
\item Since $\ell \in S^*$, it is a continuous linear functional on $S$. Hence
\[
\ell : S \to \mathbb{R},
\]
and therefore $\ell(s) \in \mathbb{R}$.

\item Since $\theta \in \mathbb{R}$, the quantity $\ell(s)-\theta$ also belongs to $\mathbb{R}$.

\item The activation function $\eta : \mathbb{R} \to \mathbb{R}$ is well defined, and therefore
\[
\eta(\ell(s)-\theta) \in \mathbb{R}.
\]

\item Because $T$ is a topological vector space, scalar multiplication is defined as a mapping
\[
\mathbb{R} \times T \to T, \qquad (a,v) \mapsto av .
\]
Consequently,
\[
\eta(\ell(s)-\theta)\, v \in T .
\]

\item Since $T$ is a vector space, finite sums of elements of $T$ again belong to $T$. Therefore expressions of the form
\[
\sum_{j=1}^{m} \eta(\ell_j(s)-\theta_j)\, v_j
\]
are well defined and take values in $T$.
\end{enumerate}

\begin{remark}
When $T=\mathbb{R}$, Theorem \ref{thm:vector_uat} reduces to the scalar-valued
UAT established in \cite{ismailov2026universal}.
If $T$ is a Banach space, the topology of uniform convergence induced by the
seminorms of $T$ coincides with the usual uniform norm topology.
\end{remark}

We first establish that vector-valued finite-rank mappings are dense in 
$C(E;T)$ under the seminorm topology of $T$.

\begin{lemma}\label{lem:finite_rank_density}
Let $E$ be a compact Hausdorff space and let $T$ be a Hausdorff LC-TVS. Then the collection of mappings of the form
\[
s \longmapsto \sum_{j=1}^{m} \psi_j(s)\, v_j,
\qquad 
\psi_j \in C(E), \; v_j \in T,
\]
is dense in $C(E;T)$ with respect to uniform convergence on $E$ 
induced by the continuous seminorms of $T$.
\end{lemma}

\begin{proof}
Fix $F \in C(E;T)$, a continuous seminorm $\rho$ on $T$, and $\varepsilon > 0$.

Since $E$ is compact and $F$ is continuous, the image set $F(E)$ is compact in $T$. In a Hausdorff LC-space, compact sets are totally bounded with respect to every continuous seminorm. 
So, there exist elements $v_1, \dots, v_m \in T$ such that
\[
F(E) \subset 
\bigcup_{j=1}^{m}
\{ t \in T : \rho(t - v_j) < \varepsilon \}.
\]

For each $j$, define
\[
V_j := \{ s \in E : \rho(F(s) - v_j) < \varepsilon \}.
\]
Then $\{V_j\}_{j=1}^m$ forms an open cover of $E$. 
Since $E$ is a compact Hausdorff space, it is paracompact. So there exists
a partition of unity $\{\psi_j\}_{j=1}^m \subset C(E)$ subordinate to the open
cover $\{V_j\}_{j=1}^m$, that is,
\[
0 \le \psi_j \le 1,
\qquad
\mathrm{supp}(\psi_j) \subset V_j,
\qquad
\sum_{j=1}^{m} \psi_j(s) = 1
\quad \text{for all } s \in E.
\]

Define
\[
G(s) := \sum_{j=1}^{m} \psi_j(s)\, v_j.
\]
Then $G \in C(E;T)$. Moreover, for every $s \in E$,
\[
\begin{aligned}
\rho\big(F(s) - G(s)\big)
&=
\rho\!\left(
\sum_{j=1}^{m} \psi_j(s)\big(F(s) - v_j\big)
\right)
\\
&\le
\sum_{j=1}^{m} \psi_j(s)\, \rho\big(F(s) - v_j\big)
< \varepsilon,
\end{aligned}
\]
since $\psi_j(s) \neq 0$ implies $s \in V_j$. 
Taking the supremum over $E$ yields
\[
\sup_{s \in E} \rho\big(F(s) - G(s)\big) < \varepsilon .
\]
\end{proof}

\begin{remark}
Lemma~\ref{lem:finite_rank_density} shows that $C(E;T)$ is generated, 
in the seminorm topology, by finite linear combinations of scalar-valued 
continuous functions with coefficients in $T$. 
If $T$ is a Banach space, the seminorm topology reduces to the usual 
uniform norm topology.
\end{remark}

Now, we next recall the scalar-valued approximation result in the topological 
vector space setting, which will serve as the key ingredient in the proof 
of the main theorem.

\begin{lemma}\label{lem:scalar_tvs_uat}
Let $S$ be a TVS with the HBEP 
and let $E \subset S$ be compact. 
Assume that $\eta : \mathbb{R} \to \mathbb{R}$ is continuous and 
not a polynomial on any open interval. 

Define
\[
\mathcal{A}_{S}^{\eta}
:=
\mathrm{span}
\left\{
s \longmapsto \eta\big(\ell(s) - \theta\big)
:
\ell \in S^*, \; \theta \in \mathbb{R}
\right\}.
\]
Then $\mathcal{A}_{S}^{\eta}$ is dense in $C(E)$.
\end{lemma}

\begin{proof}
This statement follows from the scalar UAT for NNs defined on TVS; see 
\cite{ismailov2026universal}.
\end{proof}

\begin{proof}[Proof of Theorem~\ref{thm:vector_uat}]
Let $F \in C(E;T)$, let $\rho$ be a continuous seminorm on $T$, and fix $\varepsilon > 0$.

By Lemma~\ref{lem:finite_rank_density}, there exist elements $v_1, \dots, v_m \in T$ and functions $\phi_1, \dots, \phi_m \in C(E)$ such that
\[
\sup_{s \in E}
\rho\!\left(
F(s) - \sum_{j=1}^{m} \phi_j(s)\, v_j
\right)
<
\frac{\varepsilon}{2}.
\]

Define
\[
C := \max_{1 \le j \le m} \rho(v_j).
\]

If $C=0$, then $\rho(v_j)=0$ for all $j$, and therefore
\[
\rho\!\left(F(s) - \sum_{j=1}^{m} \phi_j(s)v_j\right)
= \rho(F(s))
\]
for all $s \in E$. The above estimate implies
\[
\sup_{s\in E} \rho(F(s)) < \frac{\varepsilon}{2},
\]
and taking $G \equiv 0$ completes the proof in this case.

Assume now that $C>0$. By Lemma~\ref{lem:scalar_tvs_uat}, for each 
$j=1,\dots,m$ there exists $\psi_j \in \mathcal{A}_{S}^{\eta}$ such that
\[
\sup_{s \in E} 
|\phi_j(s) - \psi_j(s)|
<
\frac{\varepsilon}{2mC}.
\]

Define
\[
G(s)
:=
\sum_{j=1}^{m} \psi_j(s)\, v_j.
\]
By construction, $G \in \mathcal{A}_{S,T}^{\eta}$.

For each $s \in E$ we estimate
\[
\begin{aligned}
\rho\big(F(s) - G(s)\big)
&\le
\rho\!\left(
F(s) - \sum_{j=1}^{m} \phi_j(s)\, v_j
\right)
+
\rho\!\left(
\sum_{j=1}^{m} (\phi_j(s)-\psi_j(s))\, v_j
\right)
\\
&\le
\frac{\varepsilon}{2}
+
\sum_{j=1}^{m}
|\phi_j(s)-\psi_j(s)|\, \rho(v_j)
\\
&<
\frac{\varepsilon}{2}
+
\sum_{j=1}^{m}
\frac{\varepsilon}{2mC}\, \rho(v_j)
\le
\varepsilon.
\end{aligned}
\]

Taking the supremum over $E$ gives
\[
\sup_{s \in E} \rho\big(F(s)-G(s)\big) < \varepsilon,
\]
which proves the desired approximation.
\end{proof}

The preceding theorem establishes the density of shallow NNs with vector-valued coefficients in $C(E;T)$ under the seminorm topology. 
When $T$ is a Banach space, the defining family of seminorms reduces to 
the norm, and the result becomes a uniform approximation theorem with 
respect to the Banach norm.

\begin{remark}[Dual formulation]
Let $T$ be a Hausdorff LC-TVS with continuous dual $T'$. Suppose a net $\{F_\alpha\} \subset C(E;T)$ converges to $F$ uniformly on $E$ with respect to the seminorm topology of $T$. 
Then for every $t' \in T'$,
\[
\sup_{s \in E}
\big|
t'(F_\alpha(s) - F(s))
\big|
\longrightarrow 0.
\]

If, in addition, the topology of $T$ is generated by its dual, the converse implication also holds.
\end{remark}

\section{Corollaries}

In this section, we derive several corollaries of the main theorem and present examples to illustrate the theoretical results.

Explicitly, a NNs $(F_\alpha)$ converges to $F$ in $C(E;T)$ if and only if
\[
\sup_{s \in E} \rho\big(F_\alpha(s) - F(s)\big) \longrightarrow 0
\quad \text{for every continuous seminorm } \rho \text{ on } T.
\]
Whenever sequence spaces or $L^p$ spaces are involved, we assume the usual 
duality relations under standard hypotheses (for example, $(L^p)^* = L^{p'}$ 
for $1 < p < \infty$).

\begin{corollary}[Hilbert-valued approximation]
Let $S$ be a TVS possessing the HBEP, and let $E \subset S$ be compact. 
Assume $T$ is a Hilbert space and $\eta : \mathbb{R} \to \mathbb{R}$ is continuous and non-polynomial on some open interval. 

Then the approximation class
\[
\mathcal{A}_{S,T}^{\eta}
=
\mathrm{span}
\left\{
s \mapsto \eta\big(\ell(s)-\theta\big)\, v
:
\ell \in S^*, \; \theta \in \mathbb{R}, \; v \in T
\right\}
\]
is dense in $C(E;T)$ with respect to the uniform norm topology on $E$. 
In this case, the seminorm topology reduces to the norm topology induced by the Hilbert space structure.
\end{corollary}

\begin{corollary}[Function-to-function approximation]
Let $(\Omega_1,\mu_1)$ and $(\Omega_2,\mu_2)$ be measure spaces, with $1 < p < \infty$ and $1 \le q < \infty$. 
Set $S = L^p(\Omega_1)$ and $T = L^q(\Omega_2)$. Suppose $E \subset L^p(\Omega_1)$ is compact in the $L^p$ topology, and let $\eta$ satisfy the hypotheses of Theorem~\ref{thm:vector_uat}. 
Then mappings of the form
\[
f \longmapsto 
\sum_{j=1}^{m}
\eta\!\left(
\int_{\Omega_1} f(s)\,\varphi_j(s)\, d\mu_1(s)
- \theta_j
\right)
g_j,
\qquad
\varphi_j \in L^{p'}(\Omega_1), \;
g_j \in L^q(\Omega_2),
\]
are dense in $C(E; L^q(\Omega_2))$ with respect to uniform convergence on $E$.
\end{corollary}

\begin{corollary}[Sequence-to-sequence approximation]
Let $1 < p < \infty$ and $1 \le q < \infty$, and consider $S = \ell^p$, $T = \ell^q$. 
If $E \subset \ell^p$ is compact and $\eta$ satisfies the assumptions of Theorem~\ref{thm:vector_uat}, then mappings of the form
\[
s \longmapsto 
\sum_{j=1}^{m}
\eta\!\left(
\sum_{n=1}^{\infty} a_{j,n} s_n - \theta_j
\right)
v_j,
\qquad
(a_{j,n})_{n \ge 1} \in \ell^{p'}, \;
v_j \in \ell^q,
\]
are dense in $C(E;\ell^q)$ with respect to uniform convergence on $E$.
\end{corollary}

\begin{corollary}[Matrix inputs]
Let $S = \mathbb{R}^{n \times m}$ and let $T$ be a Hausdorff LC-TVS. 
If $E \subset \mathbb{R}^{n \times m}$ is compact and $\eta$ satisfies the 
conditions of Theorem~\ref{thm:vector_uat}, then mappings of the form
\[
Z \longmapsto 
\sum_{j=1}^{r}
\eta\big( \operatorname{tr}(W_j^{\top} Z) - \theta_j \big)\, v_j,
\qquad
W_j \in \mathbb{R}^{n \times m}, \;
v_j \in T,
\]
are dense in $C(E;T)$ under the seminorm topology.
\end{corollary}

\begin{corollary}[Finite-rank operator representation] Under the assumptions of Theorem~\ref{thm:vector_uat}, every $F \in C(E;T)$ admits uniform approximation on $E$ by mappings of the form
\[
s \longmapsto 
\sum_{j=1}^{m} 
\eta\big(\phi_j(s) - \theta_j\big)\, v_j,
\qquad
\phi_j \in S^*, \; v_j \in T.
\]
Consequently, NN approximants in this setting may be interpreted 
as finite-rank operators.
\end{corollary}

Now, we conclude by indicating several important classes of LC-TVS to which Theorem~\ref{thm:vector_uat} applies. 
These examples illustrate the flexibility of the abstract framework and its relevance to operator learning in infinite-dimensional settings.

\begin{example}[Smooth function spaces]
Let $\Omega \subset \mathbb{R}^d$ be open and set 
$T = C^\infty(\Omega)$. 
Endow $C^\infty(\Omega)$ with its usual Fr\'echet topology generated by the seminorms
\[
\rho_{E,\alpha}(f)
:=
\sup_{s \in E} |\partial^\alpha f(s)|,
\]
where $E \subset \Omega$ is compact and $\alpha$ is a multi-index. 
With this topology, $C^\infty(\Omega)$ is a Hausdorff LC-TVS.

By Theorem~\ref{thm:vector_uat}, NN approximants with coefficients in $C^\infty(\Omega)$ are dense in $C(E_0; C^\infty(\Omega))$ for all compact $E_0 \subset S$, with convergence measured with respect to each seminorm $\rho_{E,\alpha}$. 
This setting naturally arises in the approximation of smooth solution operators for differential equations.
\end{example}

\begin{example}[Schwartz space]
Let $T = \mathcal{S}(\mathbb{R}^d)$ denote the Schwartz space of rapidly 
decreasing smooth functions. Its standard Fr\'echet topology is generated by the seminorms
\[
\rho_{\alpha,\beta}(f)
:=
\sup_{s \in \mathbb{R}^d}
|x^\alpha \partial^\beta f(s)|,
\]
where $\alpha$ and $\beta$ are multi-indices. 
The Schwartz space is a nuclear Fr\'echet space and hence a Hausdorff LC-TVS.

The density result of Theorem~\ref{thm:vector_uat} therefore extends to 
$C(E; \mathcal{S}(\mathbb{R}^d))$, with approximation understood simultaneously with respect to all defining seminorms. 
Such a framework is particularly relevant in signal analysis, time-frequency representations, and operator models involving rapidly decaying kernels.
\end{example}

\begin{example}[Distribution spaces]
Let $\Omega \subset \mathbb{R}^d$ be open and consider $T = \mathcal{D}'(\Omega)$, the space of distributions on $\Omega$, equipped with the strong dual topology corresponding to uniform convergence on bounded subsets of the test function space $\mathcal{D}(\Omega)$. 
With this topology, $\mathcal{D}'(\Omega)$ is a Hausdorff LC-space.

If $F: E \to \mathcal{D}'(\Omega)$ is continuous with respect to the strong topology and its range is contained in a bounded subset of $\mathcal{D}'(\Omega)$, then Theorem~\ref{thm:vector_uat} ensures uniform approximation on $E$ with respect to each defining seminorm. 
This setting is well-suited to the approximation of weak or distributional 
solutions of partial differential equations.
\end{example}

\paragraph{Relation to the Banach setting}
If the output space $T$ is Banach, the LC structure is generated by a single norm. In this case, the topology of uniform convergence determined by the seminorm family coincides with the usual uniform norm topology on $C(E;T)$. 
Thus, the Banach-valued approximation theory appears as a special instance 
of the more general LC framework developed in this paper.

\section{Applications to Operator Approximation}

In this section, we illustrate how the vector-valued UAT established above applies naturally to the approximation of nonlinear
operators between function spaces. Such mappings arise frequently in the
analysis of differential equations, inverse problems, and modern operator
learning frameworks.

\subsection{Approximation of nonlinear operators}

Let $S$ and $T$ be TVS and let $E \subset S$ be compact.
Consider a nonlinear operator
\[
\mathcal{F} : E \longrightarrow T .
\]

If $\mathcal{F}$ is continuous, Theorem~\ref{thm:vector_uat} guarantees that
$\mathcal{F}$ can be approximated uniformly on $E$ by mappings of the form
\[
\mathcal{G}(s)
=
\sum_{j=1}^{m}
\eta\big(\ell_j(s)-\theta_j\big)\, v_j,
\qquad
\ell_j \in S^*, \; v_j \in T .
\]

Thus, shallow NNs with scalar activation functions and
vector-valued coefficients provide universal approximators for continuous
operators between infinite-dimensional spaces.

\subsection{Approximation of integral operators}

A particularly important class of operators in analysis and scientific
computing is given by integral operators. Let $\Omega \subset \mathbb{R}^d$
be a bounded domain and consider the operator
\[
(\mathcal{F}f)(x)
=
\int_{\Omega} K(x,s)\, f(s)\, ds,
\]
where $K : \Omega \times \Omega \to \mathbb{R}$ is a continuous kernel.

Let $S = L^p(\Omega)$ with $1 < p < \infty$ and $T = L^q(\Omega)$ with
$1 \le q < \infty$. Suppose $E \subset L^p(\Omega)$ is compact and
$\mathcal{F} : E \to L^q(\Omega)$ is continuous.

By Theorem~\ref{thm:vector_uat}, for every $\varepsilon > 0$ and every
continuous seminorm $\rho$ on $L^q(\Omega)$, there exists a NNs
\[
\mathcal{G}(f)
=
\sum_{j=1}^{m}
\eta\!\left(
\int_{\Omega} f(s)\, \varphi_j(s)\, ds
-
\theta_j
\right)
g_j(x),
\qquad
\varphi_j \in L^{p'}(\Omega), \; g_j \in L^q(\Omega),
\]
such that
\[
\sup_{f \in E}
\rho\big(\mathcal{F}(f) - \mathcal{G}(f)\big)
<
\varepsilon .
\]

Hence, NNs of the above form are capable of uniformly
approximating integral operators on compact subsets of $L^p(\Omega)$.

\subsection{Connection with neural operator learning}

The representation
\[
\mathcal{G}(f)
=
\sum_{j=1}^{m}
\eta\big(\ell_j(f)-\theta_j\big)\, v_j
\]
is structurally similar to modern neural operator architectures used in
scientific machine learning. In these models, the function $\ell_j(f)$ acts
as a sensor or measurement functional applied to the input function $f$,
while the vectors $v_j$ represent basis functions in the output space.

Consequently, Theorem~\ref{thm:vector_uat} provides a rigorous
functional-analytic foundation for shallow neural operator models,
showing that such architectures are capable of approximating continuous
operators between infinite-dimensional spaces.

\subsection{Approximation of solution operators for partial differential equations}

Many problems in mathematical physics and engineering can be formulated as
operator equations in which the solution of a partial differential equation
(PDE) depends on an input function. In such situations, one is interested in
approximating the associated \emph{solution operator}.

Let $\Omega \subset \mathbb{R}^d$ be a bounded domain and consider a PDE of the form
\[
\mathcal{L}u = f \quad \text{in } \Omega,
\]
with suitable boundary conditions, where $\mathcal{L}$ is a differential
operator and $f$ represents the input forcing term.

Under standard assumptions, this problem defines a solution operator
\[
\mathcal{S} : S \longrightarrow T,
\qquad
f \longmapsto u,
\]
where $S$ is a function space containing admissible inputs (for example
$L^2(\Omega)$ or $H^{-1}(\Omega)$) and $T$ is a function space describing
the solution (for example $H^1_0(\Omega)$ or $C(\overline{\Omega})$).

If the solution operator $\mathcal{S}$ is continuous on a compact subset
$E \subset S$, Theorem~\ref{thm:vector_uat} implies that $\mathcal{S}$
can be approximated uniformly on $E$ by NNs of the form
\[
\mathcal{G}(f)
=
\sum_{j=1}^{m}
\eta\big(\ell_j(f)-\theta_j\big)\, v_j,
\qquad
\ell_j \in S^*, \; v_j \in T .
\]

In particular, when $S = L^p(\Omega)$ with $1 < p < \infty$, the functionals
$\ell_j$ admit the representation
\[
\ell_j(f) = \int_{\Omega} f(s)\, \varphi_j(s)\, ds,
\qquad
\varphi_j \in L^{p'}(\Omega).
\]

Thus the approximating operators take the form
\[
\mathcal{G}(f)
=
\sum_{j=1}^{m}
\eta\!\left(
\int_{\Omega} f(s)\,\varphi_j(s)\, ds
-
\theta_j
\right)
g_j(x),
\qquad
g_j \in T.
\]

Consequently, shallow neural operator architectures are capable of
approximating solution operators of partial differential equations on
compact subsets of the input space.

Such approximation results provide a theoretical justification for modern
operator-learning frameworks, in which neural networks are trained to map
input functions to the corresponding solutions of PDEs.

\section{Conclusion}

In this paper, we established a UAT for shallow NNs whose inputs belong to a TVS and whose outputs take values in a Hausdorff LC-TVS. The proposed framework shows that NNs with scalar activation functions and vector-valued coefficients generate finite neural expansions that are dense in $C(E;T)$ for compact subsets $E \subset S$, where convergence is understood with respect to the seminorm topology of the target space $T$.

This result extends the scalar-valued UAT for NNs defined on TVS and naturally includes Banach and Hilbert-valued approximation as special cases. As a consequence, classical uniform approximation results for vector-valued mappings arise as particular instances of a more general locally convex approximation framework. The theory, therefore, provides a functional-analytic foundation for shallow NN models acting between infinite-dimensional spaces.

We also illustrated the applicability of the theoretical results to several important settings. In particular, the framework allows the approximation of nonlinear operators between function spaces, including integral operators and solution operators of partial differential equations. These examples demonstrate that the proposed NN architecture can serve as a universal approximator for mappings that arise in operator learning and scientific computing.

Several directions for future research remain open. It would be of interest to derive quantitative approximation rates under additional regularity assumptions, to investigate extensions to deeper NN architectures within the same locally convex framework, and to explore variants involving stochastic inputs or operator-valued activation mechanisms.

\bigskip

\noindent\textbf{Declaration of competing interest}

The author declares that he has no competing financial interests or personal relationships that could influence the reported work in this paper.

\bibliographystyle{elsarticle-num}
\bibliography{Ref}
\end{document}